\documentclass[11pt]{amsart}
\usepackage[margin=1.1in]{geometry}
\usepackage{amsmath,amssymb,amsthm,mathtools}
\usepackage{enumitem}
\usepackage{booktabs}
\usepackage{listings}
\usepackage{xcolor}
\usepackage[colorlinks=true,linkcolor=blue!50!black,citecolor=green!40!black,urlcolor=blue!60!black]{hyperref}

\newtheorem{theorem}{Theorem}[section]
\newtheorem{lemma}[theorem]{Lemma}
\newtheorem{proposition}[theorem]{Proposition}
\newtheorem{corollary}[theorem]{Corollary}
\newtheorem{conjecture}{Conjecture}
\newtheorem*{conjecture2p}{Conjecture 2'}
\newtheorem{mainthm}{Theorem}

\theoremstyle{definition}
\newtheorem{definition}[theorem]{Definition}
\newtheorem{remark}[theorem]{Remark}

\newtheorem{problem}{Problem}

\newcommand{\Q}{\mathbb{Q}}
\newcommand{\Z}{\mathbb{Z}}
\newcommand{\N}{\mathbb{N}}
\newcommand{\Qnn}{\mathbb{Q}_{\ge 0}}
\newcommand{\tr}{\operatorname{tr}}

\makeatletter
\@namedef{subjclassname@2020}{\textup{2020} Mathematics Subject Classification}
\makeatother

\title{A Collatz-type conjecture on the set of rational numbers\\ (revised and extended version)}
\author{Mohammad Javaheri}
\address{Department of Mathematics and Statistics, Siena University\\
515 Loudon Road, Loudonville, NY 12211, USA}
\email{mjavaheri@siena.edu}
\date{October 28, 2026}

\begin{document}

\begin{abstract}
Let $\theta(x)=(x-1)/3$ if $x\ge1$ and $\theta(x)=2x/(1-x)$ if $0\le x<1$. We conjecture that the $\theta$-orbit of every nonnegative rational number ends in $0$. A weaker conjecture asserts that no map in the semigroup generated by $3x+1$ and $x/(x+2)$, other than the generators' powers, has a positive rational fixed point; equivalently, the only matrices in the semigroup generated by $\left(\begin{smallmatrix}3&1\\0&1\end{smallmatrix}\right)$ and $\left(\begin{smallmatrix}1&0\\1&2\end{smallmatrix}\right)$ with an integer eigenvalue are their pure powers. We prove sharp two-sided bounds for the smallest eigenvalue of every word in the more general semigroups attached to the maps $\varphi_{a,b}$, which explain why the pair $\{2,3\}$ is distinguished: it is the only pair with $(a-1)(b-1)=2$. Combining these bounds with congruences for the trace, we determine the possible integer eigenvalues of words, and we prove that $\theta$ has no rational $1$-circuits, no rational $2$-circuits and no even rational circuits. We show that, for each $k$, a random $k$-syllable word with exponents at most $M$ has an integer eigenvalue with probability at most $2k(k+\log_2 2k)/M$. Finally, we verify by computer that $\theta$ has no positive rational cycle of length at most $36$.
\end{abstract}

\maketitle

%======================================================================
\section{Introduction}
%======================================================================

For rational numbers $a,b\ge1$, let
\[
\varphi_{a,b}(x)=\begin{cases}\dfrac{x-1}{a}, & x\ge1,\\[2mm] \dfrac{bx}{1-x}, & 0\le x<1.\end{cases}
\]
By the \emph{orbit} of $x$ under $\varphi_{a,b}$ we mean the set $\{x,\varphi_{a,b}(x),\varphi_{a,b}^2(x),\dots\}$. Throughout, $\Qnn$ denotes the set of nonnegative rational numbers. We are interested in the following question.

\medskip
\noindent\textbf{Question.} For which rational numbers $a,b\ge1$ does the orbit of every $x\in\Qnn$ under $\varphi_{a,b}$ end in $0$?
\medskip

The case $a=b=1$ gives the slow continued fraction algorithm \cite{AF,AO}
\[
\psi(x)=\begin{cases}x-1, & x\ge1,\\ x/(1-x), & 0\le x<1,\end{cases}
\]
and every orbit of $\psi$ starting in $\Qnn$ ends in $0$ (see Theorem~\ref{thm:TO}(iii) for a more general statement).

\begin{definition}
A map $\varphi:\Qnn\to\Qnn$ with $\varphi(0)=0$ has the \emph{terminating orbits} (TO) property if for every $x\in\Qnn$ there exists $n\ge1$ with $\varphi^n(x)=0$.
\end{definition}

\begin{theorem}\label{thm:TO}
Let $a=a_1/a_2$ and $b=b_1/b_2$ with $a_1,a_2,b_1,b_2\in\N$ and $\gcd(a_1,a_2)=\gcd(b_1,b_2)=1$.
\begin{enumerate}[label=(\roman*)]
\item $\varphi_{a,b}$ has the TO property if and only if $\varphi_{b,a}$ does.
\item If $\varphi_{a,b}$ has the TO property, then $\gcd(a_1,b_1)=1$, and $a_2$ and $b_2$ have the same (possibly empty) set of prime factors. In particular $a\in\N$ if and only if $b\in\N$.
\item $\varphi_{a,1}$ (equivalently $\varphi_{1,a}$) has the TO property if and only if $a\in\N$.
\end{enumerate}
\end{theorem}

For integers $a,b>1$ we make the following conjecture.

\begin{conjecture}\label{conj:1}
Let $a,b\in\N$ with $a,b>1$.
\begin{enumerate}[label=(\alph*)]
\item The map $\theta=\varphi_{3,2}$ (equivalently $\varphi_{2,3}$) has the TO property.
\item The map $\varphi_{a,b}$ does not have the TO property if $\{a,b\}\neq\{2,3\}$.
\end{enumerate}
\end{conjecture}

Part (a) has been verified for all rational numbers $m/n$ with $1\le m,n\le 2^{24}$ \cite{Yoon}. By Theorem~\ref{thm:TO}(ii), part (b) holds whenever $\gcd(a,b)>1$. The TO property can fail in two ways.
\begin{enumerate}[label=(\roman*)]
\item \emph{Periodic orbits exist.} For example, $\varphi_{3,4}$ has the periodic orbit $\{1/3,2\}$.
\item \emph{Infinite orbits exist.} For example, every reduced $p/q>0$ with $p+q$ odd has an infinite $\varphi_{2,4}$-orbit (Proposition~\ref{prop:24}); for $x=4$ the orbit is $4,\tfrac32,\tfrac14,\tfrac43,\tfrac16,\tfrac45,16,\dots$.
\end{enumerate}
To prove Conjecture~\ref{conj:1}(a) one must rule out both scenarios for $\varphi_{3,2}$; to prove part (b) one must show that one of them occurs for $\varphi_{a,b}$. In Section~\ref{sec:comp} we confirm part (b) by exhibiting rational cycles for $40$ of the $107$ coprime pairs $2\le a<b\le 20$.

Equivalently, one can study $\mathcal O_{a,b}$, the orbit of $0$ under the semigroup generated by
\begin{equation}\label{eq:rs}
r(x)=ax+1\quad\text{and}\quad s(x)=\frac{x}{x+b},
\end{equation}
the two inverse branches of $\varphi_{a,b}$, and ask for which $a,b\ge1$ we have $\mathcal O_{a,b}=\Qnn$. View $\mathcal O_{a,b}$ as a directed graph with an edge from $x$ to $r(x)$ and to $s(x)$. Since $r$ and $s$ are injective with disjoint images $[1,\infty)$ and $[0,1)$, every vertex has in-degree at most $1$; deleting $0$ we obtain a binary tree rooted at $1=r(0)$. Hence every nonnegative rational appears at most once in the tree, and Conjecture~\ref{conj:1}(a) claims that each appears exactly once. In \cite{J} it is proved that $\mathcal O_{a,b}$ is dense for all $a,b\ge1$, which gives the following.

\begin{theorem}[\cite{J}]\label{thm:dense}
Let $a,b\in\Q$ with $a,b\ge1$. Every interval $I\subseteq[0,\infty)$ contains a rational number whose $\varphi_{a,b}$-orbit ends in $0$.
\end{theorem}

\subsection*{The Collatz conjecture}
Our results are motivated by results on the Collatz conjecture, which states that the orbit of every positive integer under
\[
\kappa(n)=\begin{cases} n/2, & n \text{ even},\\ (3n+1)/2, & n\text{ odd},\end{cases}
\]
ends in $1$; see \cite{Lag85,Wir} for surveys. Krasikov \cite{Kra} showed that the number $\pi(x)$ of integers in $[1,x]$ satisfying the conjecture is at least $x^{3/7}$ for large $x$, later improved to $x^{0.84}$ \cite{KL}. Terras \cite{Ter76,Ter79} showed that the set of integers with infinite stopping time has asymptotic density zero, and Tao \cite{Tao} proved that almost all orbits (in logarithmic density) attain almost bounded values. Numerically, the conjecture has been verified for all $n<2^{71}$ \cite{Ba}.

If $u(x)=x/2$ and $v(x)=(3x+1)/2$ are the branches of $\kappa$, every finite sequence in $\{u,v\}$ occurs as the symbolic dynamics of some integer. In Proposition~\ref{prop:symbolic} we determine all infinite sequences in the branches $(x-1)/3$ and $2x/(1-x)$ of $\varphi_{3,2}$ that are realized by points of $[0,\infty)$.

The weak Collatz conjecture states that $\{1,2\}$ is the only cycle of $\kappa$ on the positive integers. Steiner \cite{Ste78} defined a $k$-circuit as a cycle that splits into $k$ blocks, each consisting of an increasing run of odd integers followed by a decreasing run of even integers, and proved that $\{1,2\}$ is the only $1$-circuit of $\kappa$. For the $Qx+1$ map ($n\mapsto n/2$ or $(Qn+1)/2$) he showed \cite{Ste81} that for $Q=5$ the $1$-circuits are $\{1,3,8,4,2\}$ and $\{13,33,83,208,104,52,26\}$, and that for $Q=7$ the only $1$-circuit is the trivial one $\{1,4,2\}$. We make the analogous definitions for $\theta$.

\begin{definition}\label{def:circuit}
A sequence $r_1,\dots,r_n$ of positive real numbers is a \emph{cycle} of $\theta$ if $r_{i+1}=\theta(r_i)$ for $i<n$ and $\theta(r_n)=r_1$; it is \emph{rational} if all $r_i\in\Q$. Since $\theta(1)=0$, no cycle contains $1$. A cycle is a \emph{$k$-circuit} if, read cyclically, it consists of $k$ blocks, each made of a run of numbers less than $1$ followed by a run of numbers greater than $1$. A $k$-circuit is \emph{even} if each of its $2k$ runs has even length.
\end{definition}

Since $\theta=s^{-1}$ on $[0,1)$ and $\theta=r^{-1}$ on $(1,\infty)$ (with $a=3$, $b=2$ in \eqref{eq:rs}), a $k$-circuit whose runs have lengths $n_1,m_1,\dots,n_k,m_k$ consists of fixed points of the cyclic conjugates of $s^{n_1}r^{m_1}\cdots s^{n_k}r^{m_k}$. Hence:

\begin{conjecture}\label{conj:2}
The map $\theta=\varphi_{3,2}$ has no positive rational cycles. Equivalently, no element of the semigroup generated by $r$ and $s$, other than the powers $r^n$ and $s^n$, has a positive rational fixed point.
\end{conjecture}

The map $\theta$ does have irrational cycles, e.g.\ $\{1/\sqrt3,\sqrt3+1\}$; in fact every finite cyclic itinerary is realized by a cycle (Proposition~\ref{prop:symbolic}), and Conjecture~\ref{conj:2} claims that none of these cycles is rational. By contrast, the Collatz map extended to rationals with odd denominator has a rational cycle for every finite itinerary \cite{Lag90}.

\subsection*{Matrix semigroups}
The map $A=\left(\begin{smallmatrix}\alpha&\beta\\\gamma&\delta\end{smallmatrix}\right)\mapsto \widehat A(x)=(\alpha x+\beta)/(\gamma x+\delta)$ is a homomorphism from invertible $2\times2$ real matrices to linear fractional maps. For integers $a,b\ge2$ let
\begin{equation}\label{eq:RS}
R_a=\begin{pmatrix}a&1\\0&1\end{pmatrix},\qquad S_b=\begin{pmatrix}1&0\\1&b\end{pmatrix},\qquad R=R_3,\quad S=S_2,
\end{equation}
so that $\widehat{R_a}(x)=ax+1$ and $\widehat{S_b}(x)=x/(x+b)$. The semigroups $\langle R_a,S_b\rangle$ and $\langle \widehat{R_a},\widehat{S_b}\rangle$ are free, and a word that is not a pure power of a generator has integer eigenvalues if and only if the corresponding map has a positive rational fixed point (Proposition~\ref{prop:corr}). Thus Conjecture~\ref{conj:2} is equivalent to:

\begin{conjecture2p}
The only matrices in $\langle R,S\rangle$ with an integer eigenvalue are $R^n$ and $S^n$, $n\ge1$.
\end{conjecture2p}

Every word in $\langle R_a,S_b\rangle$ that is not a pure power is conjugate to a \emph{$k$-syllable word}
\[
A=R_a^{m_1}S_b^{n_1}\cdots R_a^{m_k}S_b^{n_k},\qquad m_i,n_i\ge1,
\]
and conjugate matrices have the same eigenvalues. Our first main result is a sharp bound for the smallest eigenvalue $\lambda_-(A)$.

\begin{mainthm}\label{thmA}
Let $a,b\ge2$ be integers with $\delta=(a-1)(b-1)\ge2$, and let $A$ be a $k$-syllable word in $\langle R_a,S_b\rangle$. Then $A$ has two distinct positive eigenvalues $\lambda_-<\lambda_+$, and
\[
\prod_{i=1}^k \frac{\delta}{1+(\delta-1)\bigl[1-(1-a^{-m_i})(1-b^{-n_i})\bigr]}\;<\;\lambda_-\;<\;\prod_{i=1}^k\frac{\delta}{1+(\delta-1)a^{-m_i}}.
\]
In particular $\lambda_-<\delta^k$, and $\lambda_-\to\delta^k$ as all exponents tend to infinity. For $(a,b)=(2,2)$ one has $\lambda_-=1$ for every word.
\end{mainthm}

Among pairs of integers $a,b\ge2$, the value $\delta=1$ occurs only for $(2,2)$, and $\delta=2$ only for $\{a,b\}=\{2,3\}$. For $\delta=2$ and $k=1$ Theorem~\ref{thmA} gives $1<\lambda_-<2$, so there are no rational $1$-circuits, while for $\delta\ge3$ the interval $(\cdot,\delta)$ contains integers and $1$-circuits do occur (e.g.\ $R_3S_4$ has eigenvalues $2$ and $6$, giving the cycle $\{1/3,2\}$ of $\varphi_{3,4}$). Conjecture~\ref{conj:1} thus singles out precisely the smallest non-degenerate value of $(a-1)(b-1)$.

Combining Theorem~\ref{thmA} with congruences for the trace (Lemma~\ref{lem:cong}) we determine the shape of possible integer eigenvalues in $\langle R,S\rangle$ (Theorem~\ref{thm:structure}) and obtain:

\begin{mainthm}\label{thmB}
The map $\theta$ has
\begin{enumerate}[label=(\roman*)]
\item no rational $1$-circuits;
\item no rational $2$-circuits;
\item no even rational circuits.
\end{enumerate}
Moreover, a rational $k$-circuit of $\theta$ whose runs of numbers less than $1$ all have even length requires $k\ge3$.
\end{mainthm}

For $k\ge1$ and $M\ge1$ let $\Lambda_{k,M}=\Lambda_{k,M}(a,b)$ denote the set of $k$-syllable words with $1\le m_i,n_i\le M$, and let $\Omega_{k,M}\subseteq\Lambda_{k,M}$ be the subset of words with an integer eigenvalue.

\begin{mainthm}\label{thmC}
Let $a,b\ge2$ be integers with $\delta=(a-1)(b-1)\ge2$. Then
\[
\frac{|\Omega_{k,M}|}{|\Lambda_{k,M}|}\le \frac{2k\bigl(k\log_2\delta+\log_2(2k(\delta-1))\bigr)}{M}.
\]
For $(a,b)=(3,2)$ the bound is $2k(k+\log_2 2k)/M$. In particular, the density is zero for each fixed $k$, and also when $k=k(M)=o(\sqrt M)$.
\end{mainthm}

The analogous statement for general integer matrices is that almost all of them have no integer eigenvalue \cite{MW}.

Finally, in Section~\ref{sec:comp} we prove by an exhaustive computation that $\theta$ has no positive rational cycle of length at most $36$, and that no $k$-syllable word with $2\le k\le8$ and exponents in a range depending on $k$ has an integer eigenvalue.

\subsection*{Undecidability}
In general, deciding whether a finitely generated matrix semigroup contains a matrix with integer eigenvalues may be impossible. The following argument, due to B.~Poonen \cite{Poonen}, shows that it is undecidable whether a semigroup generated by two $255\times255$ integer matrices contains a matrix \emph{all} of whose eigenvalues are integers. Let $n=15$ and $p=17$. Let $f\in\Z[x]$ be monic of degree $p$ with Galois group $S_p$ (e.g.\ $f=x^{17}-x-1$ \cite{Osada}), and let $C$ be a $p\times p$ integer matrix with characteristic polynomial $f$. For $n\times n$ integer matrices $A,B$, the semigroup generated by $A\otimes C$ and $B\otimes C$ consists of the matrices $W\otimes C^m$, with $W\in\langle A,B\rangle$ and $m\ge1$. If $\nu$ is a root of $f$, then $\Q(\nu^m)\subseteq\Q(\nu)$ has degree $1$ or $p$. Degree $1$ is impossible: if $\nu^m=c\in\Q$ then all roots of $f$ are of the form $\zeta\nu$ with $\zeta^m=1$, so the splitting field of $f$ lies in the solvable extension $\Q(\zeta_m,\nu)$, contradicting $\operatorname{Gal}(f)=S_p$ with $p\ge5$. Hence every eigenvalue $\nu^m$ of $C^m$ has degree $p>n$. If $\lambda\neq0$ is an eigenvalue of $W$, then $\lambda\nu^m\notin\Z$, since otherwise $\nu^m\in\Q(\lambda)$ would have degree at most $n$. So all eigenvalues of $W\otimes C^m$ are integers if and only if $W$ is nilpotent, i.e.\ if and only if $\langle A,B\rangle$ contains the zero matrix. Mortality is undecidable for pairs of $15\times15$ integer matrices \cite{CHHN} (see also \cite{BT,CK,Pat}).

\subsection*{Organization}
Section~\ref{sec:TO} proves Theorem~\ref{thm:TO} and discusses infinite orbits. Section~\ref{sec:matrix} sets up the matrix correspondence and the symbolic dynamics of $\varphi_{3,2}$. Section~\ref{sec:bounds} proves Theorem~\ref{thmA}. Section~\ref{sec:trace} gives trace formulas and congruences. Section~\ref{sec:circuits} proves the structure theorem and Theorem~\ref{thmB}. Section~\ref{sec:density} proves Theorem~\ref{thmC}. Section~\ref{sec:comp} reports computations, and Section~\ref{sec:problems} lists open problems.

%======================================================================
\section{The TO property}\label{sec:TO}
%======================================================================

\begin{lemma}\label{lem:induction}
Suppose $\varphi_{a,b}$ has the TO property, and let $\mathcal S$ be a statement about positive rationals such that
\begin{enumerate}[label=(\roman*)]
\item $\mathcal S(1)$ holds, and
\item if $\mathcal S(x)$ holds then so do $\mathcal S(ax+1)$ and $\mathcal S(x/(x+b))$.
\end{enumerate}
Then $\mathcal S(x)$ holds for every positive rational $x$.
\end{lemma}

\begin{proof}
Let $x>0$ be rational. By the TO property $\varphi_{a,b}^n(x)=0$ for some $n$. The only preimages of $0$ are $0$ and $1$, so $\varphi_{a,b}^j(x)=1$ for some $j\ge0$, i.e.\ $x=w(1)$ for a word $w$ of length $j$ in the inverse branches $ax+1$ and $x/(x+b)$. Induct on $j$.
\end{proof}

\begin{proof}[Proof of Theorem~\ref{thm:TO}]
(i) Every orbit that reaches $0$ from a positive number first reaches $1$, so it suffices to show that $1$ lies in every positive rational orbit of $\varphi_{a,b}$ if and only if it does for $\varphi_{b,a}$. Let $\nu(x)=1/x$. A direct computation shows $\varphi_{a,b}=\nu\circ\varphi_{b,a}\circ\nu^{-1}$ on $(0,\infty)\setminus\{1\}$, and both maps send $1$ to $0$. Hence $\nu$ maps the positive part of each $\varphi_{b,a}$-orbit onto the positive part of a $\varphi_{a,b}$-orbit, and $1$ lies in one if and only if it lies in the other.

(ii) Suppose $\varphi_{a,b}$ has the TO property. For a positive rational written in lowest terms as $p/q$,
\[
ax+1=\frac{a_1p+a_2q}{a_2q},\qquad \frac{x}{x+b}=\frac{b_2p}{b_2p+b_1q},
\]
and the reduced numerator and denominator of these fractions divide the displayed ones.

Suppose $d=\gcd(a_1,b_1)>1$ and let $\mathcal S(p/q)$ be the statement $\gcd(pq,d)=1$. Clearly $\mathcal S(1)$ holds. If $\mathcal S(p/q)$ holds, then $\gcd(a_1p+a_2q,d)=\gcd(a_2q,d)=1$, since $d\mid a_1$, $\gcd(a_1,a_2)=1$ and $\gcd(q,d)=1$; and $\gcd(a_2q,d)=1$. So $\mathcal S(ax+1)$ holds, and similarly $\mathcal S(x/(x+b))$. Lemma~\ref{lem:induction} gives $\mathcal S(x)$ for all $x$, which fails for $x=d$. Hence $\gcd(a_1,b_1)=1$.

Next, suppose a prime $f$ divides $a_2$ but not $b_2$, and let $\mathcal S'(p/q)$ be $f\nmid p$. Then $\mathcal S'(1)$ holds; if $f\nmid p$, then $a_1p+a_2q\equiv a_1p\not\equiv0\pmod f$ (as $\gcd(a_1,a_2)=1$), and $f\nmid b_2p$. By Lemma~\ref{lem:induction}, $f\nmid p$ for every positive rational $p/q$, which fails for $x=f$. Symmetrically (using denominators), every prime factor of $b_2$ divides $a_2$. Finally $a\in\N\iff a_2=1\iff b_2=1\iff b\in\N$.

(iii) If $a\notin\N$, then $a_2$ has a prime factor while $b_2=1$ has none, so by (ii) $\varphi_{a,1}$ is not TO. Let $a\in\N$ and $x=p/q$ in lowest terms. If $x\ge1$, then $\varphi_{a,1}(x)=(p-q)/(aq)$ has reduced numerator at most $p-q<p$. If $0<x<1$, then $\varphi_{a,1}(x)=p/(q-p)$ is reduced, with the same numerator and smaller denominator. So the pair (numerator, denominator) strictly decreases in the lexicographic order of $\N\times\N$, which is a well-order; hence the orbit reaches $0$. By (i), the same holds for $\varphi_{1,a}$.
\end{proof}

\begin{proposition}\label{prop:24}
If $p/q>0$ is in lowest terms and $p+q$ is odd, then the $\varphi_{2,4}$-orbit of $p/q$ is infinite.
\end{proposition}

\begin{proof}
Let $\eta(p/q)=p+q$. If $p\ge q$ then $\varphi_{2,4}(p/q)=(p-q)/(2q)$; since $p-q$ is odd, this fraction is reduced and $\eta$ is unchanged. If $p<q$ then $\varphi_{2,4}(p/q)=4p/(q-p)$; since $q-p$ is odd, this is reduced and $\eta$ increases by $2p\ge2$. In both cases the parity of $\eta$ is preserved. Repeated application of $x\mapsto (x-1)/2$ eventually lands in $[0,1)$, so the second case occurs infinitely often, and $\eta\to\infty$.
\end{proof}

For $x=4$ this gives $4,\tfrac32,\tfrac14,\tfrac43,\tfrac16,\tfrac45,16,\tfrac{15}2,\tfrac{13}4,\tfrac98,\tfrac1{16},\dots$ with $\eta=5,5,5,7,7,9,17,\dots$. Note that $\eta$ is not monotone along general $\varphi_{2,4}$-orbits (e.g.\ $\varphi_{2,4}(1/5)=1$); the parity invariant is essential.

%======================================================================
\section{The matrix correspondence and symbolic dynamics}\label{sec:matrix}
%======================================================================

\begin{lemma}\label{lem:free}
For integers $a,b\ge2$ the semigroups $\langle \widehat{R_a},\widehat{S_b}\rangle$ and $\langle R_a,S_b\rangle$ are free.
\end{lemma}

\begin{proof}
(Compare Klarner \cite{Kla} for affine maps.) Since $A\mapsto\widehat A$ is a homomorphism, it suffices to treat the maps. Suppose some map has two different representations, and take one of minimal length. The representations then begin with different letters, so $r t'=s t''$ for possibly empty words $t',t''$. But $rt'(x)\ge1>st''(x)$ for all $x\ge0$, a contradiction.
\end{proof}

\begin{proposition}\label{prop:corr}
The map $A\mapsto\widehat A$ is a bijection from $\langle R_a,S_b\rangle$ onto $\langle\widehat{R_a},\widehat{S_b}\rangle$. For $A\in\langle R_a,S_b\rangle$ the following are equivalent:
\begin{enumerate}[label=(\roman*)]
\item $A$ has a rational eigenvalue;
\item $A$ has integer eigenvalues;
\item $\widehat A$ has a positive rational fixed point, or $A$ is a pure power of $R_a$ or $S_b$.
\end{enumerate}
\end{proposition}

\begin{proof}
Suppose $\widehat A=\widehat B$ with $A\ne B$, chosen with minimal total length. Then $A=cB$ for a scalar $c$, and $A,B$ begin with different letters, say $R_aA'=cS_bB'$. The vector $w=A'(1,1)^T$ has positive entries, so $R_aw=(aw_1+w_2,w_2)^T$ lies strictly below the diagonal, while $cS_bB'(1,1)^T=c(v_1,v_1+bv_2)^T$ (with $v=B'(1,1)^T>0$) lies strictly above it. This is a contradiction.

Let $A=\left(\begin{smallmatrix}\alpha&\beta\\\gamma&\delta\end{smallmatrix}\right)$. Its characteristic polynomial $x^2-(\alpha+\delta)x+(\alpha\delta-\beta\gamma)$ is monic with integer coefficients, so a rational eigenvalue is an integer, and then so is the other; thus (i)$\iff$(ii). The fixed points of $\widehat A$ satisfy $\gamma x^2+(\delta-\alpha)x-\beta=0$, whose discriminant equals that of the characteristic polynomial. If $A$ is not a pure power, then $\beta,\gamma>0$, so the two fixed points have product $-\beta/\gamma<0$: exactly one is positive, and it is rational if and only if the discriminant is a square, i.e.\ if and only if (ii) holds. Pure powers of $R_a$ and $S_b$ have the integer eigenvalues $\{a^n,1\}$ and $\{1,b^n\}$.
\end{proof}

\begin{corollary}
Conjectures~\ref{conj:2} and 2' are equivalent.
\end{corollary}

We now describe which infinite itineraries are realized by $\varphi_{3,2}$.

\begin{definition}
For $x\ge0$ let $d_k(x)=1$ if $\varphi_{3,2}^k(x)\ge1$ and $d_k(x)=0$ otherwise. The sequence $\sigma(x)=(d_0(x),d_1(x),\dots)$ is the \emph{orbit-generating sequence} of $x$. A binary sequence is \emph{terminating} if it is eventually $0$. Let $\mathcal B$ be the set of binary sequences with infinitely many zeros.
\end{definition}

Every $\sigma(x)$ lies in $\mathcal B$, since iterating $(x-1)/3$ eventually leads below $1$.

\begin{lemma}\label{lem:interval}
Let $T_0,\dots,T_k\in\{(x-1)/3,\ 2x/(1-x)\}$. There is an interval $[\alpha,\beta)\subseteq[0,\infty)$ (possibly $\beta=\infty$) that $T_k\circ\cdots\circ T_0$ maps increasingly onto $[0,\infty)$, and on which the itinerary of $\varphi_{3,2}$ is $T_0,\dots,T_k$.
\end{lemma}

\begin{proof}
For $k=0$ use $[1,\infty)$ or $[0,1)$. If $f=T_k\circ\cdots\circ T_0$ maps $[\alpha,\beta)$ increasingly onto $[0,\infty)$ and $f(\gamma)=1$, then $T_{k+1}\circ f$ maps $[\alpha,\gamma)$ (if $T_{k+1}=2x/(1-x)$) or $[\gamma,\beta)$ (if $T_{k+1}=(x-1)/3$) increasingly onto $[0,\infty)$.
\end{proof}

\begin{proposition}\label{prop:symbolic}
The map $\sigma:[0,\infty)\to\mathcal B$ is a bijection.
\end{proposition}

\begin{proof}
Let $t=(d_0,d_1,\dots)\in\mathcal B$, and let $I_k=[\alpha_k,\beta_k)$ be the interval of Lemma~\ref{lem:interval} for $(d_0,\dots,d_k)$. The intervals are nested, $\alpha_k\le\alpha_{k+1}<\beta_{k+1}\le\beta_k$, and $\beta_{k+1}=\beta_k$ exactly when $d_{k+1}=1$. Let $x_0=\lim\alpha_k$. If $x_0=\beta_k$ for some $k$, then $\beta_i=x_0$ for all $i\ge k$, forcing $d_i=1$ for all $i>k$, which contradicts $t\in\mathcal B$. So $x_0\in\bigcap_kI_k$ and $\sigma(x_0)=t$.

Suppose $\bigcap_kI_k$ contains an interval. By Theorem~\ref{thm:dense} it contains a rational number whose orbit ends in $0$, so $t$ is terminating. Pick an irrational $x$ in the intersection. Then $\sigma(x)=t$ is eventually $0$, so from some time $k$ on the orbit of $x$ stays in $[0,1)$; since $2y/(1-y)\ge2y$ this forces $\varphi_{3,2}^k(x)=0$. But $\varphi_{3,2}^k$ is a linear fractional map with rational coefficients on the relevant interval, so its zero is rational, a contradiction. Hence $\bigcap_kI_k=\{x_0\}$.
\end{proof}

In particular every purely periodic sequence in $\mathcal B$ is the orbit-generating sequence of exactly one point (the all-ones sequence corresponds to $\infty$). For $\varphi_{1,1}$ the points with purely periodic orbit-generating sequences are the positive quadratic irrationals with negative conjugate \cite{AO}. Conjecture~\ref{conj:2} asserts that no positive rational number has a purely periodic orbit-generating sequence for $\varphi_{3,2}$.

%======================================================================
\section{Bounds for the smallest eigenvalue}\label{sec:bounds}
%======================================================================

Fix integers $a,b\ge2$ and put $\delta=(a-1)(b-1)$. For a syllable $(m,n)$ let
\[
c=\frac{a^m-1}{a-1},\qquad d=\frac{b^n-1}{b-1},\qquad R_a^m=\begin{pmatrix}a^m&c\\0&1\end{pmatrix},\quad S_b^n=\begin{pmatrix}1&0\\d&b^n\end{pmatrix},
\]
and define the adjugates
\[
\bar R^m:=a^mR_a^{-m}=\begin{pmatrix}1&-c\\0&a^m\end{pmatrix},\qquad \bar S^n:=b^nS_b^{-n}=\begin{pmatrix}b^n&0\\-d&1\end{pmatrix}.
\]
The key observation is a conservation law for the linear form $\ell(x,y)=x+(b-1)y$.

\begin{lemma}\label{lem:conservation}
If $(x',y')^T=\bar S^n\bar R^m(x,y)^T$, then
\[
x'=b^n(x-cy),\qquad y'=a^my-d(x-cy),\qquad \ell(x',y')=x+\kappa y,\quad \kappa=(b-1)a^m-c.
\]
Moreover $\kappa-(b-1)=\dfrac{(a^m-1)(\delta-1)}{a-1}$.
\end{lemma}

\begin{proof}
The formulas for $x',y'$ are immediate. Then $\ell(x',y')=(b^n-(b-1)d)(x-cy)+(b-1)a^my=x-cy+(b-1)a^my$, since $b^n-(b-1)d=1$. Finally $\kappa-(b-1)=(b-1)(a^m-1)-c=(a^m-1)\bigl(b-1-\tfrac1{a-1}\bigr)$.
\end{proof}

For $(a,b)=(3,2)$ the law reads $x'+y'=x+\tfrac{3^m+1}{2}y$.

\begin{proof}[Proof of Theorem~\ref{thmA}]
Let $A=R_a^{m_1}S_b^{n_1}\cdots R_a^{m_k}S_b^{n_k}$. Each\[R_a^mS_b^n=\begin{pmatrix}a^m+cd&cb^n\\d&b^n\end{pmatrix}\]has positive entries, hence so does $A$. Its discriminant $(A_{11}-A_{22})^2+4A_{12}A_{21}$ is positive, so the eigenvalues $\lambda_-<\lambda_+$ are real and distinct, and positive since $\tr A,\det A>0$. By the Perron--Frobenius theorem there is an eigenvector $v=(x_0,y_0)^T$ for $\lambda_+$ with $x_0,y_0>0$.

Define $(x_i,y_i)^T=\bar S^{n_i}\bar R^{m_i}(x_{i-1},y_{i-1})^T$ for $i=1,\dots,k$. Since $\bar S^{n_k}\bar R^{m_k}\cdots\bar S^{n_1}\bar R^{m_1}=\det(A)\,A^{-1}$, we have
\[
(x_k,y_k)^T=\det(A)A^{-1}v=\frac{\det A}{\lambda_+}\,v=\lambda_-\,v.
\]
All $x_i,y_i$ are positive, by downward induction from $i=k$, because\[(x_{i-1},y_{i-1})^T=a^{-m_i}b^{-n_i}R_a^{m_i}S_b^{n_i}(x_i,y_i)^T\]and $R_a^{m_i}S_b^{n_i}$ has positive entries.

Let $\rho_i=\ell(x_i,y_i)/\ell(x_{i-1},y_{i-1})$. Since $\ell$ is linear, $\prod_{i=1}^k\rho_i=\ell(\lambda_-v)/\ell(v)=\lambda_-$. By Lemma~\ref{lem:conservation}, with $t=y_{i-1}/x_{i-1}$,
\[
\rho_i=\frac{1+\kappa t}{1+(b-1)t},
\]
which is strictly increasing in $t$ when $\delta>1$. The positivity conditions $x_i>0$ and $y_i>0$ say precisely that
\[
\frac{d}{a^m+cd}<t<\frac1c.
\]
Using $c+\kappa=(b-1)a^m$ and $(a-1)(c+b-1)=a^m-1+\delta$, one computes
\[
\rho\Bigl(\frac1c\Bigr)=\frac{\delta a^m}{a^m-1+\delta}=\frac{\delta}{1+(\delta-1)a^{-m}},
\]
\[
\rho\Bigl(\frac{d}{a^m+cd}\Bigr)=\frac{\delta a^mb^n}{a^mb^n+(\delta-1)(a^m+b^n-1)}=\frac{\delta}{1+(\delta-1)\bigl[1-(1-a^{-m})(1-b^{-n})\bigr]}.
\]
Multiplying the resulting strict inequalities for $\rho_1,\dots,\rho_k$ proves the bounds. Both bounds tend to $\delta$ per syllable as $m,n\to\infty$.

If $(a,b)=(2,2)$ then $\delta=1$, $\kappa=b-1$, every $\rho_i=1$, and $\lambda_-=1$.
\end{proof}

For $(a,b)=(3,2)$ we write
\[
g(m,n)=\frac{1}{2-(1-3^{-m})(1-2^{-n})}=\frac{2^n3^m}{(2^n+1)(3^m+1)-2},\qquad h(m)=\frac{3^m}{3^m+1},
\]
so that Theorem~\ref{thmA} reads
\begin{equation}\label{eq:bounds32}
2^k\prod_{i=1}^kg(m_i,n_i)<\lambda_-<2^k\prod_{i=1}^kh(m_i).
\end{equation}
The function $g$ is increasing in both variables, with $g(1,1)=\tfrac35$, $g(1,2)=\tfrac23$ and $g(2,1)=\tfrac9{14}$. Hence $(6/5)^k<\lambda_-<2^k$ for every $k$-syllable word in $\langle R,S\rangle$.

\begin{corollary}\label{cor:22}
For $(a,b)=(2,2)$, every word in $\langle R_2,S_2\rangle$ that is not a pure power has a positive rational fixed point. Hence every finite cyclic itinerary is realized by a positive rational cycle of $\varphi_{2,2}$.
\end{corollary}

\begin{proof}[Proof of Theorem~\ref{thmB}(i)]
A rational $1$-circuit gives a word $R^mS^n$ with integer eigenvalues (Proposition~\ref{prop:corr}), but by \eqref{eq:bounds32} with $k=1$ we have $1<\tfrac65<\lambda_-<2$.
\end{proof}

%======================================================================
\section{Trace formulas and congruences}\label{sec:trace}
%======================================================================

In this section $A=R^{m_1}S^{n_1}\cdots R^{m_k}S^{n_k}$ with $m_i,n_i\ge1$; indices are taken modulo $k$, and $[\,\cdot\,]$ denotes the indicator function.

\begin{proposition}\label{prop:trace}
\begin{align}
2^k\tr(A)&=\sum_{F\subseteq\{1,\dots,k\}}3^{\sum_{j\in F}m_j}\prod_{i=1}^k\Bigl(2^{n_i}+(-1)^{[i\in F]+[i+1\in F]}\Bigr), \label{eq:tr1}\\
2^k\tr(A)&=\sum_{E\subseteq\{1,\dots,k\}}2^{\sum_{i\in E}n_i}\prod_{j=1}^k\Bigl(1+(-1)^{[j-1\in E]+[j\in E]}\,3^{m_j}\Bigr). \label{eq:tr2}
\end{align}
In particular $\tr(R^mS^n)=\tfrac12(3^m+1)(2^n+1)$.
\end{proposition}

\begin{proof}
We have $R^m=3^mP+Q$ and $S^n=2^nP'+Q'$ with the rank-one projections
\[
P=\binom10\bigl(1,\tfrac12\bigr),\quad Q=\binom{-1/2}{1}(0,1),\quad P'=\binom01(1,1),\quad Q'=\binom1{-1}(1,0).
\]
Expanding $A$, its trace is a sum over choices $X_i\in\{P,Q\}$ and $Y_i\in\{P',Q'\}$. For rank-one factors, $\tr(X_1Y_1\cdots X_kY_k)=\prod_i\langle X_i|Y_i\rangle\langle Y_i|X_{i+1}\rangle$, where $\langle U|V\rangle$ is the row vector of $U$ applied to the column vector of $V$. These pairings are
\[
\langle P|P'\rangle=\langle P|Q'\rangle=\tfrac12,\quad\langle Q|P'\rangle=1,\quad\langle Q|Q'\rangle=-1,
\]
\[
\langle P'|P\rangle=\langle Q'|P\rangle=1,\quad\langle P'|Q\rangle=\tfrac12,\quad\langle Q'|Q\rangle=-\tfrac12.
\]
Write $f_i=[X_i=P]$ and $e_i=[Y_i=P']$. The total power of $\tfrac12$ is $\sum_if_i+\sum_i(1-f_{i+1})=k$, and the sign is $(-1)^{\sum_i(1-e_i)[(1-f_i)+(1-f_{i+1})]}$. For fixed $F=\{i:f_i=1\}$ the sum over the $e_i$ factorizes as $\prod_i\bigl(2^{n_i}+(-1)^{f_i+f_{i+1}}\bigr)$, which gives \eqref{eq:tr1}. Reducing the sign exponent modulo $2$ gives $\sum_jf_j(e_{j-1}+e_j)$; summing over the $f_j$ first gives \eqref{eq:tr2}.
\end{proof}

\begin{lemma}\label{lem:cong}
\[
\tr(A)\equiv\prod_{i=1}^k(1+m_i)\pmod2,\qquad
\tr(A)\equiv\prod_{i=1}^k\bigl(2^{n_i+1}-1\bigr)\equiv\begin{cases}1,&\text{all } n_i\text{ even},\\0,&\text{otherwise}\end{cases}\pmod3.
\]
In particular $\tr(A)$ is odd if and only if all $m_i$ are even, and $\tr(A)\not\equiv2\pmod3$.
\end{lemma}

\begin{proof}
Modulo $3$: $(3^m-1)/2=1+3+\dots+3^{m-1}\equiv1$, so $R^m\equiv\left(\begin{smallmatrix}0&1\\0&1\end{smallmatrix}\right)$ and $R^mS^n\equiv\binom11(2^n-1,\,2^n)$. The trace of a product of rank-one matrices $u\,v_i^T$ is $\prod_iv_i^Tu=\prod_i(2^{n_i+1}-1)$, and $2^{n+1}-1$ is $\equiv1$ or $\equiv0\pmod 3$ according as $n$ is even or odd.

Modulo $2$: $S^n\equiv\left(\begin{smallmatrix}1&0\\1&0\end{smallmatrix}\right)$ and $(3^m-1)/2\equiv m$, so $R^mS^n\equiv\binom{1+m}{1}(1,0)$, and $\tr(A)\equiv\prod_i(1+m_i)$.
\end{proof}

%======================================================================
\section{Integer eigenvalues and circuits}\label{sec:circuits}
%======================================================================

Let $A=R^{m_1}S^{n_1}\cdots R^{m_k}S^{n_k}$ with $m_i,n_i\ge1$, and set $M=\sum_im_i\ge k$ and $N=\sum_in_i\ge k$, so that $\det A=2^N3^M$.

\begin{theorem}\label{thm:structure}
Suppose $A$ has integer eigenvalues $\lambda_-<\lambda_+$. Then:
\begin{enumerate}[label=(\alph*)]
\item $\lambda_-=2^\alpha3^\beta$ and $\lambda_+=2^{N-\alpha}3^{M-\beta}$ with $(6/5)^k<\lambda_-<2^k$. Hence $\alpha<k\le N$ and $\beta<k\le M$, so $6\mid\lambda_+$ and $\tr(A)\equiv\lambda_-\pmod6$.
\item $\lambda_-$ is odd if and only if all $m_i$ are even.
\item If $3\nmid\lambda_-$, then $\lambda_-=4^j$ with $j\ge1$, and all $n_i$ are even. If $3\mid\lambda_-$, then some $n_i$ is odd.
\end{enumerate}
Consequently exactly one of the following holds:
\begin{itemize}
\item all $m_i$ are even, some $n_i$ is odd, and $\lambda_-=3^\beta$ with $(9/7)^k<3^\beta<2^k$;
\item all $n_i$ are even, some $m_i$ is odd, and $\lambda_-=4^j$ with $(4/3)^k<4^j<2^k$;
\item some $m_i$ and some $n_i$ are odd, and $6\mid\lambda_-$.
\end{itemize}
\end{theorem}

\begin{proof}
(a) Two integers with product $2^N3^M$ have the stated form, and the bounds are \eqref{eq:bounds32}. Since $2^\alpha\le\lambda_-<2^k$ and $3^\beta\le\lambda_-<2^k$, we get $\alpha,\beta<k$.

(b) By (a), $\tr(A)\equiv\lambda_-\pmod2$; now apply Lemma~\ref{lem:cong}.

(c) If $3\nmid\lambda_-$, then $\lambda_-=2^\alpha\equiv\tr(A)\pmod3$, and $\tr(A)\not\equiv2$ by Lemma~\ref{lem:cong}. Hence $\alpha$ is even, $\tr(A)\equiv1$, and all $n_i$ are even. If $3\mid\lambda_-$, then $3\mid\tr(A)$, so some $n_i$ is odd. Finally, $\lambda_-\neq1$ by (a).

For the three cases: if all $m_i$ and all $n_i$ were even, then $\lambda_-$ would be odd and prime to $3$, hence $1$, which is impossible. The sharper lower bounds follow from $g(m,n)\ge g(1,2)=\tfrac23$ when $n\ge2$, and $g(m,n)\ge g(2,1)=\tfrac9{14}$ when $m\ge2$.
\end{proof}

\begin{proof}[Proof of Theorem~\ref{thmB}(iii) and the final statement]
An even rational circuit would give a word with all $m_i$ and $n_i$ even and integer eigenvalues, which Theorem~\ref{thm:structure} excludes. If all $n_i$ are even, then $4^j\in((4/3)^k,2^k)$ with $j\ge1$, which forces $k\ge3$.
\end{proof}

\begin{proof}[Proof of Theorem~\ref{thmB}(ii)]
Let $k=2$. By Theorem~\ref{thm:structure}, $\lambda_-\in(1.44,4)$ is of the form $2^\alpha3^\beta$, so $\lambda_-\in\{2,3\}$. The value $2$ is excluded by (c), so $\lambda_-=3$, and by (b) $m_1,m_2$ are even. Put $U=3^{m_1}$ and $V=3^{m_2}$, which are powers of $9$, and $X=2^{n_1}$ and $Y=2^{n_2}$. By \eqref{eq:tr1},
\[
4\tr(A)=(1+UV)(X+1)(Y+1)+(U+V)(X-1)(Y-1).
\]
Multiplying the condition $\tr(A)=3+\det(A)/3=3+XYUV/3$ by $12$ and rearranging gives
\begin{equation}\label{eq:2circ}
(Y-1)K=L,\quad K=(X-3)UV-3(X-1)(U+V)-3(X+1),\quad L=(2X+6)UV+6X-30.
\end{equation}
The equation is symmetric under $X\leftrightarrow Y$ and under $U\leftrightarrow V$, so we may assume $X\le Y$ and $U\le V$. Since $L>0$, we need $K>0$.
\begin{itemize}
\item $X=2$: $K=-UV-3(U+V)-9<0$.
\item $X=4$: $K=UV-9(U+V)-15$ and $L=14UV-6$. If $U=9$ then $K=-96<0$. If $U,V\ge81$ then $K\ge\frac{5088}{6561}UV$, so $Y-1<18.06$, i.e.\ $c:=Y-1\in\{3,7,15\}$. For $c\in\{3,7\}$ the equation reads $(c-14)UV-9c(U+V)-15c+6=0$, whose left side is negative. For $c=15$ it reads $(U-135)(V-135)=18444$. With $U,V\in\{81,729,\dots\}$ the factors are both $-54$ (product $2916$) or both at least $594$ (product at least $352836$), so there is no solution.
\item $X=8$: $K=5UV-21(U+V)-27$ and $L=22UV+18$. If $U=9$ then $K=24V-216$, so $V\ge81$ and $Y-1=8.25+1800/(24V-216)\in(8.25,9.3]$; this forces $Y=10$, which is not a power of $2$. If $U,V\ge81$ then $K>4.47\,UV$, so $Y-1<4.93<X-1$, a contradiction.
\item $X=16$: $K=13UV-45(U+V)-51$ and $L=38UV+66$. For $U=V=9$, $Y-1=3144/192\notin\Z$. For $U=9<V$, $Y-1=4.75+2232/(72V-456)\le5.17<X-1$. For $U,V\ge81$, $K>11.8\,UV$ and $Y-1<3.3<X-1$.
\item $X\ge32$: from $U+V\le\tfrac29UV$ and $1\le UV/81$ we get $K\ge\tfrac{8X-64}{27}UV>0$ and $L\le\bigl(2X+6+\tfrac{6X}{81}\bigr)UV$. Hence $Y-1\le\frac{27(2X+6+6X/81)}{8X-64}\le10.2<X-1$, a contradiction.\qedhere
\end{itemize}
\end{proof}

\begin{remark}
For $k=3$, Theorem~\ref{thm:structure} leaves $\lambda_-\in\{3,4,6\}$, and for $k=4$ it leaves $\lambda_-\in\{3,4,6,9,12\}$. For each fixed $k$ and each candidate $\lambda$, the equation $\lambda^2-\lambda\tr(A)+\det A=0$ is, by \eqref{eq:tr1}, a linear equation in $\{2,3\}$-units. The theory of $S$-unit equations \cite{ESS} suggests that its solutions lie in finitely many degenerate families; carrying this out would settle Conjecture~$2'$ for each fixed $k$.
\end{remark}

%======================================================================
\section{Asymptotic density}\label{sec:density}
%======================================================================

\begin{proof}[Proof of Theorem~\ref{thmC}]
By Lemma~\ref{lem:free}, $|\Lambda_{k,M}|=M^{2k}$. Let $A\in\Omega_{k,M}$, and put $\varepsilon_i=1-(1-a^{-m_i})(1-b^{-n_i})\le a^{-m_i}+b^{-n_i}\le2^{-m_i}+2^{-n_i}$. Then $\lambda_-$ is an integer (Proposition~\ref{prop:corr}), and $\lambda_-<\delta^k$ by Theorem~\ref{thmA}, so $\lambda_-\le\delta^k-1$. The lower bound of Theorem~\ref{thmA} gives
\[
\prod_{i=1}^k\bigl(1+(\delta-1)\varepsilon_i\bigr)>\frac{1}{1-\delta^{-k}}.
\]
Taking logarithms and using $\log(1+u)\le u$ and $-\log(1-u)\ge u$, we obtain $(\delta-1)\sum_i\varepsilon_i>\delta^{-k}$, hence
\[
\sum_{i=1}^k\bigl(2^{-m_i}+2^{-n_i}\bigr)>\frac{\delta^{-k}}{\delta-1}.
\]
So at least one of the $2k$ exponents $x\in\{m_i,n_i\}$ satisfies $2^{-x}>\delta^{-k}/(2k(\delta-1))$, i.e.\ $x<T_k:=k\log_2\delta+\log_2(2k(\delta-1))$. For each of the $2k$ positions there are fewer than $T_kM^{2k-1}$ tuples with a small exponent in that position, so $|\Omega_{k,M}|\le2kT_kM^{2k-1}$.
\end{proof}

\begin{remark}
The same argument shows that if every exponent of a $k$-syllable word exceeds $T_k$, then the word has no integer eigenvalue. This holds for every pair $a,b\ge2$ with $\delta\ge2$, not just for $(3,2)$.
\end{remark}

%======================================================================
\section{Computational results}\label{sec:comp}
%======================================================================

All computations use exact integer arithmetic: 128-bit integers, or arithmetic modulo $2^{61}-1$ with every hit confirmed exactly. The programs are listed in the Appendix. As a check, each search was run on pairs where cycles are known to exist: for $(a,b)=(3,4)$ it finds $R_3S_4$ with eigenvalues $2$ and $6$, and for $(a,b)=(2,2)$ it flags every word, as Corollary~\ref{cor:22} predicts.

\subsection{Short cycles of $\theta$}
A positive rational cycle of $\theta$ of length $L$ gives a word of length $L$ in $R,S$ with integer eigenvalues. After a cyclic rotation, the word begins with $R$ and ends with $S$. We tested all $2^{L-2}$ such words for each $L\le36$ ($2^{35}-1\approx3.4\times10^{10}$ words), checking whether $\tr(A)^2-4\det A$ is a perfect square.

\begin{proposition}\label{prop:L36}
The map $\theta=\varphi_{3,2}$ has no positive rational cycle of length at most $36$.
\end{proposition}

\subsection{Fixed number of syllables}
For fixed $k$ we tested $\lambda^2-\lambda\tr(A)+\det A=0$ for every $\{2,3\}$-smooth $\lambda$ with $2\le\lambda<2^k$, over all exponent vectors with $1\le m_i,n_i\le B$:
\begin{center}
\begin{tabular}{cccc}
\toprule
$k$ & $B$ & words tested & integer eigenvalues\\
\midrule
2 & 150 & 506{,}250{,}000 & none\\
3 & 28 & 481{,}890{,}304 & none\\
4 & 12 & 429{,}981{,}696 & none\\
5 & 7 & 282{,}475{,}249 & none\\
6 & 5 & 244{,}140{,}625 & none\\
7 & 4 & 268{,}435{,}456 & none\\
8 & 3 & 43{,}046{,}721 & none\\
\bottomrule
\end{tabular}
\end{center}
For $k=2$ this agrees with Theorem~\ref{thmB}(ii).

\subsection{Other pairs $(a,b)$}
We iterated $\varphi_{a,b}$ on every reduced $p/q$ with $p,q\le30$, stopping at $0$, at a repeated value, or when $p+q>10^{40}$.

\begin{proposition}\label{prop:pairs}
Among the $107$ coprime pairs $2\le a<b\le20$ with $(a,b)\ne(2,3)$, exactly $40$ have a rational cycle reached from some $p/q$ with $p,q\le30$. For these pairs, and for their mirror images $(b,a)$ by Theorem~\ref{thm:TO}(i), Conjecture~\ref{conj:1}(b) holds. Each entry lists the pair, the cycle length, and the element of the cycle with least height $p+q$:
\begin{center}\small
\begin{tabular}{llll}
$(2,5)\!:\,3,\ 1/2$ & $(2,7)\!:\,16,\ 1/4$ & $(2,9)\!:\,5,\ 6/1$ & $(2,11)\!:\,4,\ 1/2$ \\
$(2,15)\!:\,5,\ 1/12$ & $(2,17)\!:\,7,\ 7/8$ & $(2,19)\!:\,15,\ 1/2$ & $(3,4)\!:\,2,\ 2/1$ \\
$(3,5)\!:\,3,\ 3/1$ & $(3,7)\!:\,6,\ 5/1$ & $(3,8)\!:\,5,\ 6/1$ & $(3,10)\!:\,3,\ 8/1$ \\
$(3,11)\!:\,16,\ 7/5$ & $(3,13)\!:\,15,\ 11/1$ & $(3,14)\!:\,3,\ 2/1$ & $(3,16)\!:\,9,\ 14/1$ \\
$(4,5)\!:\,10,\ 3/2$ & $(4,7)\!:\,31,\ 59/10$ & $(4,9)\!:\,3,\ 7/2$ & $(4,11)\!:\,4,\ 9/2$ \\
$(4,13)\!:\,3,\ 1/2$ & $(4,15)\!:\,13,\ 13/2$ & $(5,6)\!:\,21,\ 1/3$ & $(5,8)\!:\,2,\ 2/1$ \\
$(5,9)\!:\,37,\ 7/3$ & $(5,12)\!:\,8,\ 9/7$ & $(5,14)\!:\,3,\ 4/1$ & $(6,7)\!:\,5,\ 4/9$ \\
$(7,9)\!:\,4,\ 7/5$ & $(7,12)\!:\,2,\ 2/1$ & $(8,9)\!:\,2,\ 3/1$ & $(8,19)\!:\,3,\ 3/4$ \\
$(9,10)\!:\,5,\ 8/7$ & $(9,13)\!:\,4,\ 1/2$ & $(9,14)\!:\,5,\ 11/15$ & $(9,16)\!:\,2,\ 2/1$ \\
$(11,20)\!:\,2,\ 2/1$ & $(12,13)\!:\,11,\ 2/1$ & $(15,16)\!:\,2,\ 4/1$ & $(18,19)\!:\,7,\ 2/1$ \\
\end{tabular}
\end{center}
\end{proposition}

For the remaining $67$ pairs, which include $(2,13)$, $(3,17)$, $(5,7)$, $(5,11)$, $(6,11)$ and $(7,8)$, no cycle was detected, and most orbits appear to escape to infinity. The fates of all $979$ reduced $p/q$ with $p,q\le40$ (at most $2\cdot10^4$ steps, height cap $10^{60}$) are:
\begin{center}
\begin{tabular}{lccc}
\toprule
$(a,b)$ & reach $0$ & enter a cycle & exceed height $10^{60}$\\
\midrule
$(2,3)$ & 979 & 0 & 0\\
$(2,5)$ & 741 & 238 & 0\\
$(2,7)$ & 819 & 160 & 0\\
$(3,4)$ & 752 & 227 & 0\\
$(3,5)$ & 534 & 113 & 332\\
$(3,7)$ & 14 & 15 & 950\\
$(4,5)$ & 897 & 82 & 0\\
$(4,7)$ & 308 & 66 & 605\\
$(5,6)$ & 22 & 37 & 920\\
$(5,7)$ & 7 & 0 & 972\\
$(6,7)$ & 290 & 63 & 626\\
\bottomrule
\end{tabular}
\end{center}
By Theorem~\ref{thm:TO}(i) the pairs $(a,b)$ and $(b,a)$ behave identically; we checked that the data agree.

%======================================================================
\section{Open problems}\label{sec:problems}
%======================================================================

\begin{problem}
Extend Theorem~\ref{thmB}(ii) to $k=3$, where $\lambda_-\in\{3,4,6\}$ by Theorem~\ref{thm:structure}.
\end{problem}

\begin{problem}
Does the proportion of words of length $L$ in $\langle R,S\rangle$ with an integer eigenvalue tend to $0$ as $L\to\infty$? Theorem~\ref{thmC} does not apply directly, since a random word of length $L$ has about $L/4$ syllables with geometrically distributed exponents.
\end{problem}

\begin{problem}
For which coprime $(a,b)$ with $\delta\ge3$ does $\varphi_{a,b}$ have rational $1$-circuits? By Theorem~\ref{thmA} this reduces to exponential Diophantine equations in $a^m$ and $b^n$.
\end{problem}

\begin{problem}
For pairs such as $(5,7)$, prove the existence of an infinite orbit. A height potential in the spirit of Proposition~\ref{prop:24} seems natural; cancellation in $\varphi_{a,b}$ comes only from $\gcd(p-q,a)$ and $\gcd(q-p,b)$.
\end{problem}

\subsection*{Acknowledgements}
I would like to thank Peter Yoon for the experimental verification of Conjecture~\ref{conj:1}(a) for rational numbers $m/n$ with $1\le m,n\le2^{24}$, Bjorn Poonen for reading an earlier version of the manuscript, suggesting corrections and providing the undecidability argument in the introduction, and the reviewer whose careful reading improved earlier versions and motivated the results of Section~\ref{sec:TO} and Proposition~\ref{prop:symbolic}. This revised and extended version was prepared by an AI assistant (Claude).

%======================================================================
\appendix
\section*{Appendix: programs}
%======================================================================

\noindent\textbf{Exhaustive search over words of length $\le L$} (Proposition~\ref{prop:L36}); run as $65$ jobs indexed by a $6$-letter prefix:
\begin{lstlisting}
#include <stdio.h>
#include <stdlib.h>
#include <math.h>
typedef __int128 i128;
static int LMAX; static long long found=0, leaves=0;
static int isq(i128 x){ if(x<0) return 0; long double r=sqrtl((long double)x); i128 s=(i128)r;
  while(s*s>x) s--; while((s+1)*(s+1)<=x) s++; return s*s==x; }
static char wb[64];
static void dfs(i128 a,i128 b,i128 c,i128 d,int len,int last){
  if(last==1){ leaves++; i128 T=a+d, D=a*d-b*c; if(isq(T*T-4*D)){found++; wb[len]=0; printf("FOUND %s\n",wb);} }
  if(len==LMAX) return;
  wb[len]='R'; dfs(3*a, a+b, 3*c, c+d, len+1, 0);
  wb[len]='S'; dfs(a+b, 2*b, c+d, 2*d, len+1, 1);
}
/* job j in [0,64): fixed prefix R + 6 bits; words shorter than 7 handled by job 0 separately */
int main(int argc,char**argv){ LMAX=atoi(argv[1]); int j=atoi(argv[2]);
  i128 a=3,b=1,c=0,d=1; wb[0]='R'; int last=0;
  if(j==64){ /* short words len<=7 */ int save=LMAX; LMAX=7; dfs(3,1,0,1,1,0); printf("short leaves=%lld found=%lld\n",leaves,found); return 0;}
  for(int t=0;t<7;t++){ int bit=(j>>t)&1; if(t==6){} 
    if(t<6){ if(bit){ wb[t+1]='S'; i128 na=a+b,nb=2*b,nc=c+d,nd=2*d; a=na;b=nb;c=nc;d=nd; last=1;} else { wb[t+1]='R'; i128 na=3*a,nb=a+b,nc=3*c,nd=c+d; a=na;b=nb;c=nc;d=nd; last=0;} } }
  /* prefix length 7 fixed; count leaves with length>=8 only: emulate dfs but skip test at len 7 */
  /* call children directly */
  if(LMAX>7){ wb[7]='R'; dfs(3*a,a+b,3*c,c+d,8,0); wb[7]='S'; dfs(a+b,2*b,c+d,2*d,8,1);} 
  printf("job %d leaves=%lld found=%lld\n",j,leaves,found); return 0;}
\end{lstlisting}

\noindent\textbf{Fixed-$k$ search modulo $2^{61}-1$} (Section~\ref{sec:comp}):
\begin{lstlisting}
#include <stdio.h>
#include <stdlib.h>
#include <stdint.h>
typedef unsigned __int128 u128; typedef uint64_t u64;
static const u64 P=2305843009213693951ULL; /* 2^61-1 */
static u64 mm(u64 a,u64 b){ u128 r=(u128)a*b; u64 lo=(u64)(r&P), hi=(u64)(r>>61); u64 s=lo+hi; if(s>=P) s-=P; return s;}
static u64 ad(u64 a,u64 b){u64 s=a+b; if(s>=P) s-=P; return s;}
static u64 sb(u64 a,u64 b){return a>=b? a-b : a+P-b;}
typedef struct {u64 a,b,c,d;} M;
static M mul(M x,M y){M z; z.a=ad(mm(x.a,y.a),mm(x.b,y.c)); z.b=ad(mm(x.a,y.b),mm(x.b,y.d)); z.c=ad(mm(x.c,y.a),mm(x.d,y.c)); z.d=ad(mm(x.c,y.b),mm(x.d,y.d)); return z;}
static int K,B; static M Rp[512],Sp[512]; static u64 p3[4096],p2[4096];
static int ex[32]; static long long cnt=0,hits=0; static u64 cand[200]; static int nc=0;
static void rec(int pos,M A,int sm,int sn){
  if(pos==2*K){ cnt++; u64 T=ad(A.a,A.d), D=mm(p3[sm],p2[sn]);
    for(int i=0;i<nc;i++){u64 a=cand[i]; if(sb(ad(mm(a,a),D),mm(a,T))==0){hits++; printf("HIT a=%llu:",(unsigned long long)a); for(int j=0;j<2*K;j++) printf(" %d",ex[j]); printf("\n");}}
    return;}
  for(int e=1;e<=B;e++){ ex[pos]=e; M N = (pos%2==0)? mul(A,Rp[e]) : mul(A,Sp[e]);
    rec(pos+1,N, sm+((pos%2==0)?e:0), sn+((pos%2==1)?e:0)); }
}
int main(int argc,char**argv){K=atoi(argv[1]);B=atoi(argv[2]);
  p3[0]=p2[0]=1; for(int i=1;i<4096;i++){p3[i]=mm(p3[i-1],3);p2[i]=mm(p2[i-1],2);}
  M R={3,1,0,1},S={1,0,1,2},I={1,0,0,1}; Rp[0]=I;Sp[0]=I; for(int i=1;i<=B;i++){Rp[i]=mul(Rp[i-1],R);Sp[i]=mul(Sp[i-1],S);}
  for(long long a=2;a<(1LL<<K);a++){long long t=a; while(t%2==0)t/=2; while(t%3==0)t/=3; if(t==1) cand[nc++]=a;}
  M A=I; rec(0,A,0,0); printf("k=%d B=%d words=%lld hits=%lld (candidates %d)\n",K,B,cnt,hits,nc); return 0;}
\end{lstlisting}

\end{document}